\documentclass[12pt,reqno,oneside]{amsart}
\usepackage{amsmath, amssymb, amsthm}
\usepackage[utf8]{inputenc}
\usepackage{geometry}
\usepackage{tikz-cd}
\usepackage{tikz}
\usepackage{adjustbox}
\usepackage{enumitem}
\usepackage{hyperref}

\newtheorem{definition}{Definition}[section]
\newtheorem{theorem}[definition]{Theorem}
\newtheorem{proposition}[definition]{Proposition}
\newtheorem{lemma}[definition]{Lemma}
\newtheorem{corollary}[definition]{Corollary}
\newtheorem{example}[definition]{Example}
\newtheorem{remark}[definition]{Remark}

\newcommand{\B}{\mathrm{B}}
\newcommand{\Ba}{\mathcal{B}}

\newcommand{\CC}{\mathcal{C}}

\newcommand{\cat}{\mathrm{cat}}

\newcommand{\Pa}{\mathrm{P}}

\newcommand{\co}{\colon}

\newcommand{\constant}{\mathrm{const}}

\newcommand{\D}{\mathcal{D}}

\newcommand{\E}{\mathcal{E}}
\newcommand{\F}{\mathcal{F}}
\newcommand{\I}{\mathbb{I}}

\newcommand{\id}{\mathrm{id}}

\newcommand{\Ima}{\mathrm{Im}}
\newcommand{\inc}{\mathrm{i}}

\newcommand{\op}{\mathrm{op}}

\newcommand{\Sets}{\mathrm{Set}}

\newcommand{\wtilde}{\widetilde}

\newcommand{\A}{\mathcal{A}}

\newcommand{\Auto}{\mathrm{Aut}}

\title[Categorical Fiber Bundles]
{Reducing and Classifying Fiber Bundles over Small Categories}
\author[I. Carcac\'ia]{%
	Isaac Carcac\'ia-Campos
}
 \address{%
           	Isaac Carcac\'ia-Campos
            \\
              CITMAga, Departamento de Matem\'aticas, Universidade de Santiago de Compostela, 15782-SPAIN}
\email{isaaccarcacia@gmail.com}

\begin{document}
\begin{abstract}
A fiber bundle over a small category is a locally constant family of categories whose transition functors encode how a fixed fiber is transported over the base. Its Grothendieck construction assembles this data into a category over the base, while its global behavior is governed by monodromy.

For finite acyclic categories, relative beat objects provide reductions of the total category over the fixed base, leading to relative cores that exist and are unique up to isomorphism. Monodromy classifies categorical fiber bundles by non-abelian cohomology and describes their strict gauge groups and sections. We also prove a fundamental-groupoid version of Quillen's Theorem~A, which gives reductions of the base that are more general than beat reductions and preserve the classification of bundles with fixed fiber. The results are illustrated by explicit finite examples.
\end{abstract}

\maketitle

\section{Introduction}

Small categories provide discrete and combinatorial models in which
homotopical constructions can be formulated directly in terms of objects,
morphisms, and universal properties. Their relevance is not limited to the
homotopy types of their classifying spaces: their internal structure supports
explicit notions of homotopy, reduction, minimality, and directed transport.
Finite acyclic categories are particularly suitable for this approach, since
their homotopies and reductions can be constructed from finite categorical
data
\cite{Kozlov,LSTan,StrongHomTan,SimpleHom_Tan,HomotopicDistance}.
This viewpoint belongs to the tradition of using small categories as
combinatorial models for homotopy theory, exemplified by Thomason's model
structure on \(\cat\) \cite{Thomason1980}. Through the nerve and classifying-space constructions, these categorical
models could retain the topological information of familiar spaces.

In this paper, we develop a combinatorial theory of fiber bundles over small
categories. A fiber bundle over a connected category \(\Ba\) with fiber
\(\F\) is a functor \(P\co\Ba\to\cat\) such that \(P(b)\cong\F\) for
every \(b\in\Ba\) and every transition functor is an isomorphism. Its
Grothendieck construction defines a projection
\(\pi_P\co\int^{\Ba}P\to\Ba\), which is a bifibration and satisfies the
strong homotopy lifting property
\cite{DAZIN,GRAY,FibrationFoscoRiehl,Varadarajan,
carcaciacampos2026weakstrongfibrationsfunctors}.

A categorical fiber bundle may be regarded as a locally constant family of
categories over a small base: its fibers carry the same categorical
structure, while the morphisms of the base encode how that structure is
transported. The Grothendieck construction assembles this local data into a
single category equipped with a projection to the base.

Our first contribution is a relative reduction theory for functors between
finite acyclic categories. Following the theories of Stong, Tanaka, and
Cianci--Ottina
\cite{Stong,MR3024764,StrongHomTan,SimpleHom_Tan,
cianci2019fibrationsfinitetopologicalspaces,
cianci2020fiberbundlesalexandroffspaces},
we define beat objects of a functor as beat objects of its total category
whose universal morphisms are vertical. Their removal gives strong
deformation retractions over the base. Iterating this procedure yields
relative cores, which are unique up to isomorphism over the base.

The global structure of a categorical fiber bundle is described by its
monodromy. After identifying the fibers and normalizing along a spanning
tree, the bundle is determined by a representation
\(\rho_P\co\pi_1(\Ba,b_0)\to\Auto(\F)\). This yields a classification by
the non-abelian cohomology set \(H^1(\Ba;\Auto(\F))\), identifies the strict
gauge group with the centralizer of the monodromy subgroup, and characterizes
horizontal and arbitrary sections through strict and lax fixed points,
respectively.

Finally, we introduce a reduction of the base designed to preserve the classification of fiber bundles rather than its full strong homotopy type. We prove a fundamental-groupoid version of Quillen's Theorem~A and use it to define upper and lower fundamental-groupoid negligible objects through punctured comma categories. Their removal preserves the fundamental groupoid of the base, and beat objects are particular cases of this more general notion. Consequently, these reductions leave unchanged the classification of fiber bundles with a fixed fiber.

All the constructions are effective for finite categories presented by
generators and relations. Fundamental groups and monodromy can be computed from spanning trees,
relative cores by successive beat-object removals, and
fundamental-groupoid negligibility from punctured comma categories.

The paper is organized as follows.
Sections~\ref{sec:prelim}--\ref{sec:homotopy} contain the categorical and
homotopical preliminaries. Section~\ref{sec:beat} develops relative cores and proves that
bifibrations with minimal finite acyclic fibers are fiber bundles.
Section~\ref{sec:Monodromy_categorical_fiber_bundles} studies monodromy,
classification, strict gauge transformations, and sections.
Section~\ref{sec:classification-reductions} introduces reductions preserving
classification data. Finally, Section~\ref{sec:examples} presents explicit
examples.

\section{Preliminaries on small categories}\label{sec:prelim}

All categories considered in this paper are small, and we denote by \(\cat\)
the category of small categories and functors. For functors
\(F,G\co\CC\to\D\), we write \(\operatorname{Nat}(F,G)\) for the set of
natural transformations from \(F\) to \(G\). Sets are identified with discrete
categories and thus regarded as objects of \(\cat\).

We shall frequently describe categories by generators and relations. More
precisely, a quiver generates a free category by adjoining identities and
composites, and a category is obtained by imposing relations between parallel
paths.

\begin{example}[The category \(S\)]\label{ex:S}
Let \(S\) be the category generated by two parallel arrows with no relations:
\[
\begin{tikzcd}
x \ar[r, "f", bend left] \ar[r, "g"', bend right] & y.
\end{tikzcd}
\]
\end{example}

\begin{example}[The projective plane category \(\mathcal{P}\)]\label{ex:P}
Let \(\mathcal{P}\) be the category generated by
\[
\begin{tikzcd}
X
  \arrow[r, "f_1", bend left=49]
  \arrow[r, "f_2"', bend right=49]
&
Y
  \arrow[r, "g_1", bend left=49]
  \arrow[r, "g_2"', bend right=49]
&
Z
\end{tikzcd}
\]
subject to the relations
\[
g_1\circ f_1=g_2\circ f_2,
\qquad
g_2\circ f_1=g_1\circ f_2.
\]
\end{example}

\begin{example}[Posets]\label{ex:Sprime}
Every poset \((X,\leq)\) determines a category with one morphism
\(x\to y\) whenever \(x\leq y\). Its Hasse diagram gives a presentation in
which every pair of parallel paths is identified. For example, let \(S'\) be
the poset with objects \(x_1,x_2,y_1,y_2\) and relations
\(x_i\leq y_j\) for all \(i,j\). It is represented by
\[
\begin{tikzcd}
x_1 \arrow[d] \arrow[rd] & x_2 \arrow[ld] \arrow[d] \\
y_1                      & y_2.
\end{tikzcd}
\]
\end{example}

For a category \(\CC\), let \(\Auto(\CC)\) denote the group of invertible
functors \(\CC\to\CC\). For instance,
\(\Auto(S)\cong\mathbb{Z}_2\), where the nontrivial automorphism exchanges
\(f\) and \(g\) and fixes the objects.

\subsection{Path connectivity}

For \(m\geq 0\), let \(\I_m\) be the free category on the alternating graph
\[
0\longrightarrow 1\longleftarrow 2\longrightarrow\cdots
\longrightarrow(\longleftarrow)m.
\]
A \emph{path} in \(\CC\) is a functor \(\I_m\to\CC\). Two objects are
\emph{connected} if they are joined by a path. A category is
\emph{connected} if every pair of its objects is connected, and its connected
components are its maximal connected full subcategories.

\section{Fiber bundles over small categories}\label{sec:fb}

We regard fiber bundles over categories as locally constant
\(\cat\)-valued functors and associate to them categories over the base by
means of the Grothendieck construction.

\begin{definition}\label{def:fiber_bundle}
Let \(\Ba\) be a connected category and let \(\F\) be a category. A
\emph{fiber bundle} over \(\Ba\) with fiber \(\F\) is a functor
\(
P\co\Ba\longrightarrow\cat
\)
such that \(P(b)\cong\F\) for every \(b\in\Ba\), and \(P(f)\) is an
isomorphism for every morphism \(f\) of \(\Ba\).
\end{definition}

We use isomorphisms rather than equivalences because the applications below
concern acyclic categories, for which every equivalence is an isomorphism.
The connectedness assumption is inessential: over an arbitrary base, the
definition can be applied separately to each connected component.

\begin{definition}[Grothendieck construction]
For a functor \(P\co\Ba\to\cat\), its Grothendieck construction
\(\int^{\Ba}P\) is the category whose objects are pairs \((b,x)\), with
\(b\in\Ba\) and \(x\in P(b)\), and whose morphisms
\(
(f,\alpha)\co(b,x)\longrightarrow(b',y)
\)
consist of a morphism \(f\co b\to b'\) in \(\Ba\) and a morphism
\(\alpha\co P(f)(x)\to y\) in \(P(b')\). Composition is given by
\[
(g,\beta)\circ(f,\alpha)
 =
\bigl(g\circ f,\beta\circ P(g)(\alpha)\bigr).
\]
The associated projection
\(
\pi_P\co\int^{\Ba}P\longrightarrow\Ba
\)
is defined by \(\pi_P(b,x)=b\) and \(\pi_P(f,\alpha)=f\).
\end{definition}

\begin{example}[Trivial bundles]\label{ex:product}
If \(P=\constant_{\F}\), then
\(
\int^{\Ba}P\cong\Ba\times\F,
\)
and \(\pi_P\) is the first projection.
\end{example}

\begin{example}[Semidirect products]\label{ex:semidirect}
Let \(G\) and \(H\) be groups, regarded as one-object categories, and let
\(P\co G\to\cat\) correspond to an action
\(g\mapsto\phi_g\in\Auto(H)\). Then \(\int^G P\) is the semidirect product
\(H\rtimes G\), with multiplication
\[
(h_1,g_1)(h_2,g_2)
 =
\bigl(h_1\phi_{g_1}(h_2),g_1g_2\bigr).
\]
\end{example}

\begin{example}\label{ex:klein bottle}
Let \(S\) be the category from Example~\ref{ex:S}, and let
\(s\co S\to S\) be the automorphism exchanging \(f\) and \(g\). Define
\[
K\co S\longrightarrow\cat,
\qquad
K(x)=K(y)=S,\qquad K(f)=\id_S,\qquad K(g)=s.
\]
The category \(\int^S K\) is generated by
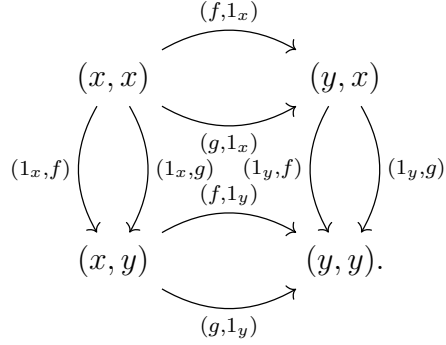
\[
\begin{tikzcd}[column sep=4em, row sep=4em]
{(x,x)}
  \arrow[r, "{(f,1_x)}", bend left]
  \arrow[r, "{(g,1_x)}"', bend right]
  \arrow[d, "{(1_x,g)}", bend left]
  \arrow[d, "{(1_x,f)}"', bend right]
&
{(y,x)}
  \arrow[d, "{(1_y,g)}", bend left]
  \arrow[d, "{(1_y,f)}"', bend right]
\\
{(x,y)}
  \arrow[r, "{(f,1_y)}", bend left]
  \arrow[r, "{(g,1_y)}"', bend right]
&
{(y,y)}.
\end{tikzcd}
\]
subject to the four relations:
\begin{itemize}
  \item $(f,1_y) \circ (1_x,f)=(f,K(f)f)=(f,f)=(1_y,f) \circ(f,1_x)$.
    \item $(f,1_y) \circ (1_x,g)=(f,K(f)g)=(f,g)=(1_y,g) \circ(f,1_x)$.
    \item $(g,1_y) \circ (1_x,f)=(g,K(g)f)=(g,g)=(1_y,g) \circ(g,1_x)$.
    \item $(g,1_y) \circ (1_x,g)=(g,K(g)g)=(g,f)=(1_y,f) \circ(g,1_x)$.
    \end{itemize}
\end{example}

\begin{example}\label{Ex:Circle_Twist_S_3}
Let \(\Ba\) be the category generated by three parallel arrows
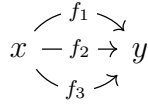
\[
\begin{tikzcd}
x
  \arrow[r, "f_1" description, bend left=49]
  \arrow[r, "f_2" description]
  \arrow[r, "f_3" description, bend right=49]
&
y
\end{tikzcd}
\]
with no relations. Define \(P\co\Ba\to\cat\) by
\[
P(x)=P(y)=[3],
\qquad
P(f_1)=\id_{[3]},
\qquad
P(f_2)=(12),
\qquad
P(f_3)=(23),
\]
where \([3]=\{1,2,3\}\) is regarded as a discrete category. The total
category is generated by
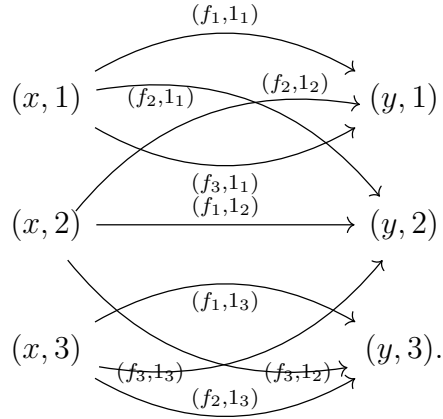
\[
\begin{tikzcd}[column sep=8em, row sep=2.2em]
{(x,1)}
  \arrow[r, "{(f_1,1_1)}", bend left]
  \arrow[r, "{(f_3,1_1)}"', bend right]
  \arrow[rd, "{(f_2,1_2)}", bend left]
&
{(y,1)}
\\
{(x,2)}
  \arrow[r, "{(f_1,1_2)}"]
  \arrow[ru, "{(f_2,1_1)}", bend left, shift right=2]
  \arrow[rd, "{(f_3,1_3)}"', bend right, shift right]
&
{(y,2)}
\\
{(x,3)}
  \arrow[ru, "{(f_3,1_2)}"', bend right, shift right]
  \arrow[r, "{(f_2,1_3)}", bend right]
  \arrow[r, "{(f_1,1_3)}"', bend left]
&
{(y,3)}.
\end{tikzcd}
\]
Since the fibers are discrete, every morphism in \(\int^{\Ba}P\) has an
identity as its second component.
\end{example}

\section{Bifibrations and coverings}\label{sec:bifib}

We recall the categorical lifting properties used throughout the paper. We
begin with opfibrations, since the covariant Grothendieck construction
naturally produces an opfibration. Our conventions follow
\cite{DAZIN,GRAY,FibrationFoscoRiehl}.

Let \(p\co\E\to\Ba\) be a functor. For \(b\in\Ba\), the \emph{fiber}
\(\E_b\) is the subcategory whose objects are mapped to \(b\) and whose
morphisms are mapped to \(1_b\). Such morphisms are called
\emph{vertical}.

\subsection{Opfibrations and fibrations}

\begin{definition}
A morphism \(\varphi\co e\to e'\) of \(\E\) is
\emph{opcartesian} if, for every morphism \(\beta\co e\to z\) and every
morphism \(\overline\alpha\co p(e')\to p(z)\) satisfying
\(\overline\alpha\circ p(\varphi)=p(\beta)\), there exists a unique
morphism \(\alpha\co e'\to z\) such that
\(\alpha\circ\varphi=\beta\) and \(p(\alpha)=\overline\alpha\).

Equivalently, every compatible diagram
\[
\begin{tikzcd}[column sep=4em, row sep=2.5em]
e
  \arrow[r, "\varphi"]
  \arrow[rd, "\beta"']
&
e'
  \arrow[d, dashed, "\alpha"]
\\
&
z
\end{tikzcd}
\qquad
\begin{tikzcd}[column sep=4em, row sep=2.5em]
p(e)
  \arrow[r, "p(\varphi)"]
  \arrow[rd, "p(\beta)"']
&
p(e')
  \arrow[d, "\overline\alpha"]
\\
&
p(z)
\end{tikzcd}
\]
admits a unique dashed lift.
\end{definition}

\begin{definition}
The functor \(p\co\E\to\Ba\) is a \emph{Grothendieck opfibration} if,
for every morphism \(f\co b\to b'\) of \(\Ba\) and every
\(e\in\E_b\), there is an opcartesian morphism
\(\varphi\co e\to e'\) such that \(p(\varphi)=f\). Such a morphism is
called an \emph{opcartesian lift} of \(f\) from \(e\).
\end{definition}

The codomain of an opcartesian lift is unique up to a unique vertical
isomorphism. Consequently, an opfibration determines transport between its
fibers, but not a preferred transport functor until lifts have been chosen.

\begin{definition}
An \emph{opcleavage} of an opfibration \(p\co\E\to\Ba\) is a choice of an
opcartesian lift
\(\varphi_{f,e}\co e\to f_!e\) for every \(f\co b\to b'\) and every
\(e\in\E_b\). These choices induce transport functors
\(f_!\co\E_b\to\E_{b'}\).

An opcleavage is \emph{closed} if the chosen lifts of identity morphisms are
identities and the chosen lift of a composite is the composite of the chosen
lifts. An opfibration equipped with a closed opcleavage is called
\emph{split}.
\end{definition}

Assuming choice, every opfibration admits an opcleavage. Its transports
define a pseudofunctor \(\Ba\to\cat\), with coherent isomorphisms
\((gf)_!\cong g_!f_!\) and \((1_b)_!\cong\id_{\E_b}\). A closed
opcleavage gives a strict functor, and the Grothendieck construction of a
strict functor has a canonical closed opcleavage.

\begin{proposition}\label{prop:functor_to_Opfibration}
For every functor \(P\co\Ba\to\cat\), the projection
\(\pi_P\co\int^{\Ba}P\to\Ba\) is a split Grothendieck opfibration.
\end{proposition}

\begin{proof}
For \(f\co b\to b'\) and \(x\in P(b)\), define
\[
\varphi_{f,x}
=
(f,1_{P(f)(x)})
\co
(b,x)\longrightarrow(b',P(f)(x)).
\]
Given \((g,\beta)\co(b,x)\to(c,z)\) with \(g=h\circ f\), there is a
unique factorization
\[
(g,\beta)
=
(h,\beta)\circ(f,1_{P(f)(x)}).
\]
Hence \(\varphi_{f,x}\) is opcartesian. These lifts preserve identities and
composition because \(P\) is a functor, so they form a closed opcleavage.
\end{proof}

The dual notions are obtained by reversing the arrows.

\begin{definition}
A morphism \(\varphi\co e'\to e\) is \emph{cartesian} if, for every
\(\beta\co z\to e\) and every
\(\overline\alpha\co p(z)\to p(e')\) satisfying
\(p(\varphi)\circ\overline\alpha=p(\beta)\), there exists a unique
\(\alpha\co z\to e'\) such that
\(\varphi\circ\alpha=\beta\) and
\(p(\alpha)=\overline\alpha\).

A functor is a \emph{Grothendieck fibration} if every morphism of the base
admits a cartesian lift with any prescribed codomain. A \emph{cleavage} is
a choice of such lifts. A \emph{bifibration} is both a Grothendieck
fibration and a Grothendieck opfibration.
\end{definition}

A cleavage of a fibration induces contravariant transport functors
\(f^*\co\E_{b'}\to\E_b\). As above, an arbitrary cleavage produces a
pseudofunctor \(\Ba^{\op}\to\cat\), while a closed cleavage produces a
strict functor.

\subsection{Bifibrations arising from adjunctions}

For an opfibration, the existence of compatible right adjoints to the
covariant transport functors provides cartesian lifts in the opposite
direction.

\begin{theorem}\label{prop:Bifibration_As_functor_adjoint}
Let \(P\co\Ba\to\cat\) be a functor. Suppose that, for every morphism
\(f\co b\to b'\), the functor \(P(f)\co P(b)\to P(b')\) admits a right
adjoint \(f^*\co P(b')\to P(b)\). Then
\(\pi_P\co\int^{\Ba}P\to\Ba\) is a bifibration.
\end{theorem}

\begin{proof}
By Proposition~\ref{prop:functor_to_Opfibration}, \(\pi_P\) is an
opfibration. Let \(f\co b\to b'\), let \(x\in P(b')\), and let
\(\epsilon^f_x\co P(f)(f^*x)\to x\) be the counit of
\(P(f)\dashv f^*\). The morphism
\[
(f,\epsilon^f_x)
\co
(b,f^*x)\longrightarrow(b',x)
\]
is cartesian.

Indeed, let \((g,\beta)\co(c,z)\to(b',x)\) and suppose that
\(g=f\circ h\). By functoriality,
\(\beta\co P(f)(P(h)(z))\to x\). The adjunction gives a unique morphism
\(\widetilde\beta\co P(h)(z)\to f^*x\) whose adjunct is \(\beta\).
Therefore
\((h,\widetilde\beta)\co(c,z)\to(b,f^*x)\) is the unique morphism over
\(h\) such that
\[
(f,\epsilon^f_x)\circ(h,\widetilde\beta)=(g,\beta).
\]
Thus \((f,\epsilon^f_x)\) is a cartesian lift of \(f\) with codomain
\((b',x)\).
\end{proof}

The right adjoints need not form a strict functor
\(\Ba^{\op}\to\cat\); their canonical coherence isomorphisms instead define
the associated pseudofunctor.

Since every isomorphism is adjoint to its inverse, we obtain the following
immediately.

\begin{corollary}\label{cor:bundle-bifibration}
If \(P\co\Ba\to\cat\) is a fiber bundle, then
\(\pi_P\co\int^{\Ba}P\to\Ba\) is a bifibration.
\end{corollary}

For a fiber bundle, the canonical opcartesian transport along
\(f\co b\to b'\) is \(P(f)\), and the cartesian transport is
\(P(f)^{-1}\). Thus fiber bundles provide particularly rigid examples of
bifibrations.

Conversely, a bifibration whose fibers are isomorphic to a fixed category
\(\F\) and whose transport functors are isomorphisms is equivalent over the
base to the Grothendieck construction of a morphism-inverting pseudofunctor.
When the chosen cleavages are closed, this pseudofunctor is strict and gives
a fiber bundle in the sense of Definition~\ref{def:fiber_bundle}. In
general, one obtains a fiber bundle only up to the standard
pseudofunctorial description; see \cite{FibrationFoscoRiehl}.

\subsection{Coverings}

Coverings are the discrete instances of the preceding constructions.

\begin{definition}
A functor \(p\co\E\to\Ba\) is a \emph{covering} if every morphism
\(f\co b\to p(e)\) admits a unique lift with codomain \(e\), and every
morphism \(g\co p(e)\to b\) admits a unique lift with domain \(e\).
Equivalently, a covering is a discrete bifibration.
\end{definition}

For a covering, the lifts are unique strictly, rather than merely up to a
unique vertical isomorphism. Consequently, no choice of cleavage is required.

\begin{proposition}\label{prop:coverings-bundles}
Let \(p\co\E\to\Ba\) be a covering. Then every fiber \(\E_b\) is
discrete, and the lifting operations define a functor
\[
P_p\co\Ba\longrightarrow\Sets,
\qquad
b\longmapsto\operatorname{Ob}(\E_b),
\]
which sends every morphism to a bijection. Conversely, if
\(Q\co\Ba\to\Sets\) sends every morphism to a bijection, then
\(\int^{\Ba}Q\to\Ba\) is a covering.
\end{proposition}

\begin{proof}
Uniqueness of lifts of identity morphisms implies that every vertical
morphism is an identity, so the fibers are discrete. For
\(f\co b\to b'\), lifting from prescribed domains gives a map
\(f_!\co\operatorname{Ob}(\E_b)\to\operatorname{Ob}(\E_{b'})\), while
lifting with prescribed codomains gives its inverse. Uniqueness also gives
\((gf)_!=g_!f_!\) and \((1_b)_!=\id\).

Conversely, if every \(Q(f)\) is a bijection, then
\((f,1_{Q(f)(x)})\) is the unique lift with domain \((b,x)\), and
\(Q(f)^{-1}\) gives the unique lift with any prescribed codomain.
\end{proof}

Thus coverings are precisely the fiber bundles with discrete fibers.

\begin{example}\label{ex:double_cover}
Let \(S\) be the category from Example~\ref{ex:S}, and define
\[
Q\co S\longrightarrow\Sets,
\qquad
Q(x)=Q(y)=\{1,2\},
\qquad
Q(f)=\id,
\qquad
Q(g)=\sigma,
\]
where \(\sigma\) exchanges \(1\) and \(2\). Then \(\int^S Q\) is
represented by
\[
\begin{tikzcd}[column sep=4em, row sep=3em]
{(x,1)}
  \arrow[r, "{(f,1_1)}"]
  \arrow[rd, "{(g,1_2)}", bend left]
&
{(y,1)}
\\
{(x,2)}
  \arrow[r, "{(f,1_2)}"']
  \arrow[ru, "{(g,1_1)}"', bend right]
&
{(y,2)}.
\end{tikzcd}
\]
This is a two-sheeted covering of \(S\), analogous to the nontrivial double
cover of the circle.
\end{example}
\section{Strong homotopy}\label{sec:homotopy}

Strong homotopy between functors can be described either by interval
categories or by zigzags of natural transformations.

\begin{definition}
Two functors \(F,G\co\CC\to\D\) are \emph{strongly homotopic}, written
\(F\simeq G\), if there exist \(m\geq 0\) and a functor
\(
H\co\CC\times\I_m\longrightarrow\D
\)
such that \(H(-,0)=F\) and \(H(-,m)=G\).
\end{definition}

Equivalently, \(F\simeq G\) if there is a finite zigzag
\[
F=F_0\Longleftrightarrow F_1
 \Longleftrightarrow\cdots
 \Longleftrightarrow F_n=G
\]
of natural transformations, where each arrow may point in either direction
\cite{LEE,LEE2}.

Let \(i_0\co\CC\to\CC\times\I_m\) be given by \(i_0(c)=(c,0)\). A functor
\(p\co\E\to\Ba\) has the \emph{right homotopy lifting property} if every
commutative diagram
\[
\begin{tikzcd}
\CC \arrow[d, "i_0"'] \arrow[r, "G"] & \E \arrow[d, "p"] \\
\CC\times\I_m \arrow[r, "H"']        & \Ba
\end{tikzcd}
\]
admits a lift \(\wtilde H\co\CC\times\I_m\to\E\).

\begin{proposition}[{\cite[Proposition 3.1 and Corollary 3.2]{Varadarajan}}]
Every bifibration has the right homotopy lifting property.
\end{proposition}

Together with Corollary~\ref{cor:bundle-bifibration}, this gives:

\begin{corollary}\label{cor:bundle-hlp}
If \(P\co\Ba\to\cat\) is a fiber bundle, then
\(
\pi_P\co\int^{\Ba}P\longrightarrow\Ba
\)
has the right homotopy lifting property.
\end{corollary}

\section{Relative cores and minimal fiber bundles}
\label{sec:beat}

The homotopy theory of finite \(T_0\)-spaces developed by Stong is based
on the successive removal of beat points and the resulting notion of a core
\cite{Stong,MR3024764}. Cianci and Ottina developed a relative reduction
theory for maps over a fixed finite space and studied its interaction with
fibrations
\cite{cianci2019fibrationsfinitetopologicalspaces}. They also studied fiber
bundles over Alexandroff spaces through a topological Grothendieck
construction
\cite{cianci2020fiberbundlesalexandroffspaces}. Independently, Tanaka
extended strong homotopy reduction from finite posets to finite acyclic
categories and \(\Delta\)-complexes
\cite{StrongHomTan,SimpleHom_Tan}.

In this section, we develop the corresponding relative theory for functors
between finite acyclic categories. Beat objects whose universal morphisms are
vertical can be removed without changing the base, and successive removals
produce a core in the slice category. We then apply Tanaka's rigidity lemma
to prove the uniqueness of relative cores and to show that every bifibration
with minimal finite acyclic fibers is isomorphic over its base to the
projection associated with a fiber bundle.

A category \(\A\) is \emph{acyclic} \cite[Chapter 10]{Kozlov} if
\(\operatorname{Hom}_{\A}(x,x)=\{1_x\}\) for every object \(x\), and
\(\operatorname{Hom}_{\A}(x,y)\neq\varnothing\) implies
\(\operatorname{Hom}_{\A}(y,x)=\varnothing\) whenever \(x\neq y\).
Thus finite acyclic categories generalize finite posets by allowing parallel
morphisms.
\subsection{Beat objects in finite acyclic categories}

For an object \(x\in\A\), let
\(\widehat{\A\downarrow x}\) denote the full subcategory of
\(\A\downarrow x\) obtained by removing \(1_x\). Its objects are the
nonidentity morphisms with codomain \(x\). Dually,
\(\widehat{x\downarrow\A}\) denotes the full subcategory obtained by
removing \(1_x\) from \(x\downarrow\A\).

\begin{definition}[{\cite[Definition 2.3]{StrongHomTan}}]
An object \(x\) of an acyclic category \(\A\) is a \emph{down beat
object} if \(\widehat{\A\downarrow x}\) has a terminal object. Equivalently,
there exists a morphism \(u\co y\to x\) such that every nonidentity
morphism \(v\co z\to x\) factors uniquely as \(v=u\circ\widetilde v\).

Dually, \(x\) is an \emph{up beat object} if
\(\widehat{x\downarrow\A}\) has an initial object. A \emph{beat object}
is either a down beat object or an up beat object.
\end{definition}

When \(\A\) is a poset, this recovers Stong's notion of a beat point
\cite{Stong}. A down beat object \(x\), witnessed by \(u\co y\to x\),
determines a retraction
\(r\co\A\to\A\setminus\{x\}\) and a natural transformation
\(\inc r\Rightarrow\id_{\A}\), where
\(\inc\co\A\setminus\{x\}\to\A\) is the inclusion. The dual
construction for an up beat object gives a natural transformation
\(\id_{\A}\Rightarrow\inc r\).

A finite acyclic category is \emph{minimal} if it has no beat objects. A
\emph{core} of \(\A\) is a minimal full subcategory obtained by
successively removing beat objects.

\begin{lemma}[Rigidity lemma {\cite[Proposition 2.2]{StrongHomTan}}]
\label{lem:rigidity}
Let \(\CC\) be a minimal finite acyclic category. If
\(F\co\CC\to\CC\) is strongly homotopic to \(\id_{\CC}\), then
\(F=\id_{\CC}\).
\end{lemma}

In particular, cores of finite acyclic categories are unique up to
isomorphism \cite[Corollary 2.2]{StrongHomTan}.

\begin{example}\label{ex:beat-points}
Let \(\D\) be the category generated by
\begin{center}
\begin{tikzcd}
a \arrow[r, "h"] &
b
  \arrow[r, "f_1" description, bend left=49]
  \arrow[r, "f_2" description]
  \arrow[r, "f_3" description, bend right=49]
&
c
\end{tikzcd}
\end{center}
subject to the relation \(f_1\circ h=f_3\circ h\). The category
\(\widehat{\D\downarrow b}\) has the unique object \(h\co a\to b\), so
\(b\) is a down beat object. Removing \(b\) gives the category generated by
\begin{center}
\begin{tikzcd}[column sep=4em]
a
  \arrow[r, "f_1\circ h=f_3\circ h", bend left=35]
  \arrow[r, "f_2\circ h"', bend right=35]
&
c.
\end{tikzcd}
\end{center}
This category has no beat objects and is therefore a core of \(\D\).
\end{example}

\subsection{Beat objects over a fixed base}
We shall work in the slice category \(\cat/\Ba\), whose objects are
functors to \(\Ba\) and whose morphisms are functors commuting with the
projections. A natural transformation between functors over \(\Ba\) is
said to be \emph{over \(\Ba\)} if all its components are vertical.
Strong homotopies and strong deformation retracts over \(\Ba\) are defined
by requiring all functors and natural transformations involved to lie over
\(\Ba\).
Let \(p\co\E\to\Ba\) be a functor. For \(b\in\Ba\), let \(\E_b\)
denote the fiber of \(p\) over \(b\).

\begin{definition}
An object \(e\in\E\) is a \emph{down beat object of \(p\)} if it is a
down beat object of \(\E\) and its universal morphism
\(u\co d\to e\) is vertical, that is, \(p(u)=1_{p(e)}\).

Dually, \(e\) is an \emph{up beat object of \(p\)} if it is an up beat
object of \(\E\) and its universal morphism is vertical. A
\emph{beat object of \(p\)} is either a down beat object or an up beat
object of \(p\).
\end{definition}

Equivalently, a down beat object of \(p\) is a down beat object of the total
category whose universal morphism also belongs to the fiber
\(\E_{p(e)}\). This is the categorical counterpart of the relative beat
points considered by Cianci and Ottina \cite{cianci2019fibrationsfinitetopologicalspaces}.

\begin{proposition}[Elementary relative reduction]
\label{prop:relative-beat-reduction}
Let \(p\co\E\to\Ba\) be a functor between acyclic categories, and let
\(e\in\E\) be a down beat object of \(p\), witnessed by a vertical
morphism \(u\co d\to e\). If
\(\inc\co\E\setminus\{e\}\to\E\) is the inclusion, then there is a
retraction \(r\co\E\to\E\setminus\{e\}\) over \(\Ba\) and a natural
transformation over \(\Ba\) from \(\inc r\) to \(\id_{\E}\).
Consequently, the restriction of \(p\) to \(\E\setminus\{e\}\) is a
strong deformation retract of \(p\) in \(\cat/\Ba\).

The dual statement holds for up beat objects.
\end{proposition}

\begin{proof}
Set \(\E'=\E\setminus\{e\}\). Define \(r(e)=d\) and \(r(x)=x\) for
\(x\neq e\). On morphisms not incident with \(e\), let \(r\) be the
identity, and set \(r(1_e)=1_d\).

If \(v\co x\to e\) is nonidentity, the universal property of
\(u\co d\to e\) gives a unique morphism
\(\widetilde v\co x\to d\) such that \(v=u\circ\widetilde v\). Define
\(r(v)=\widetilde v\). If \(v\co e\to x\), define \(r(v)=v\circ u\).

These assignments preserve composition. The only nontrivial case is a
composite \(x\xrightarrow{v}e\xrightarrow{w}y\). Writing
\(v=u\circ\widetilde v\), we obtain
\(r(wv)=wv=w u\widetilde v=r(w)r(v)\). The remaining cases follow
directly from the definition or from the uniqueness in the universal
property of \(u\). Thus \(r\) is a functor and \(r\inc=\id_{\E'}\).

Since \(p(u)=1_{p(e)}\), we have \(p(d)=p(e)\). Moreover, if
\(v=u\widetilde v\), then \(p(v)=p(\widetilde v)\), while
\(p(vu)=p(v)\). It follows that \(p\inc r=p\), so \(r\) is a morphism
over \(\Ba\).

Define \(\alpha\co\inc r\Rightarrow\id_{\E}\) by
\(\alpha_e=u\) and \(\alpha_x=1_x\) for \(x\neq e\). Naturality for a
morphism with codomain \(e\) is the defining equality
\(u r(v)=v\), and naturality for a morphism with domain \(e\) is the
equality \(r(v)=vu\). Every component of \(\alpha\) is vertical, so
\(\alpha\) is a natural transformation over \(\Ba\).
\end{proof}

\begin{definition}
A functor \(p\co\E\to\Ba\) between finite acyclic categories is
\emph{minimal over \(\Ba\)} if it has no beat objects. A
\emph{relative core} of \(p\) is a restriction
\(p_{\mathrm c}\co\E_{\mathrm c}\to\Ba\) obtained by successively
removing beat objects of the functor and which is minimal over \(\Ba\).
\end{definition}

\begin{theorem}[Existence of relative cores]
\label{thm:relative-core}
Every functor \(p\co\E\to\Ba\) between finite acyclic categories admits
a relative core. Moreover, its relative core is a strong deformation
retract of \(p\) in \(\cat/\Ba\).
\end{theorem}

\begin{proof}
If \(p\) has a beat object, remove it using
Proposition~\ref{prop:relative-beat-reduction}. Repeating the construction
gives a descending sequence of full subcategories
\(\E=\E_0\supset\E_1\supset\cdots\). Since \(\E\) is finite, the
process terminates at a restriction
\(p_{\mathrm c}\co\E_{\mathrm c}\to\Ba\) with no beat objects.

At every step, the inclusion admits a retraction over \(\Ba\), and its
composite with the retraction is connected to the identity by a natural
transformation over \(\Ba\), in one of the two possible directions.
Composing the retractions and concatenating the resulting zigzags gives a
retraction \(r\co\E\to\E_{\mathrm c}\) over \(\Ba\) such that
\(r\inc=\id_{\E_{\mathrm c}}\) and \(\inc r\) is strongly homotopic to
\(\id_{\E}\) over \(\Ba\).
\end{proof}

For an object \(e\) of a finite acyclic category, its height is the maximum
length of a sequence of nonidentity morphisms ending at \(e\). Dually, the coheight of \(e\) is the maximum length of a sequence of
nonidentity morphisms starting at \(e\).
\begin{lemma}[Relative rigidity]
\label{lem:relative-rigidity}
Let \(p\co\E\to\Ba\) be a functor between finite acyclic categories
which is minimal over \(\Ba\). If \(h\co\E\to\E\) is a functor over
\(\Ba\) and \(h\) is strongly homotopic to \(\id_{\E}\) through
functors and natural transformations over \(\Ba\), then
\(h=\id_{\E}\).
\end{lemma}

\begin{proof}
For an object \(e\in\E\), let \(\operatorname{ht}(e)\) denote its height,
that is, the maximal length of a chain of nonidentity morphisms ending at
\(e\). Since \(\E\) is finite and acyclic, this is well defined.

We first consider a natural transformation over \(\Ba\),
\(\alpha\co h\Rightarrow\id_{\E}\), and prove that \(h=\id_{\E}\).
Suppose otherwise, and choose an object \(e\) of minimal height such that
either \(h(e)\neq e\) or \(\alpha_e\neq1_e\).

The component \(\alpha_e\co h(e)\to e\) is nonidentity. Indeed, if it
were an identity, then \(h(e)=e\) and \(\alpha_e=1_e\), contrary to the
choice of \(e\). Since \(\E\) is acyclic, we have
\(\operatorname{ht}(h(e))<\operatorname{ht}(e)\). By the minimality of
the height of \(e\),
\(h(h(e))=h(e)\) and \(\alpha_{h(e)}=1_{h(e)}\).

We claim that \(\alpha_e\) is terminal in
\(\widehat{\E\downarrow e}\). Let \(g\co y\to e\) be a nonidentity
morphism. Then \(\operatorname{ht}(y)<\operatorname{ht}(e)\), so
\(h(y)=y\) and \(\alpha_y=1_y\). Naturality of \(\alpha\) gives
\(\alpha_e\circ h(g)=g\circ\alpha_y=g\). Thus \(g\) factors through
\(\alpha_e\).

To prove uniqueness, let \(k\co y\to h(e)\) satisfy
\(\alpha_e\circ k=g\). Every object involved below has height smaller
than that of \(e\), and hence \(h\) acts identically on these objects and
on the morphisms between them. More explicitly, naturality of \(\alpha\)
and the equalities
\(\alpha_y=1_y=\alpha_{h(e)}\) imply \(h(k)=k\). Moreover,
\(h(\alpha_e)\) is an endomorphism of \(h(e)\), because
\(h(h(e))=h(e)\), and therefore \(h(\alpha_e)=1_{h(e)}\) by
acyclicity. Applying \(h\) to the equality
\(\alpha_e\circ k=g\) gives
\(k=h(g)\). Hence the factorization through \(\alpha_e\) is unique.

Therefore \(e\) is a down beat object of \(\E\), witnessed by
\(\alpha_e\). Since \(\alpha\) is a natural transformation over
\(\Ba\), all its components are vertical, and in particular
\(p(\alpha_e)=1_{p(e)}\). Thus \(e\) is a down beat object of \(p\),
contradicting the minimality of \(p\). Consequently, \(h=\id_{\E}\).

The dual argument applies to a natural transformation over \(\Ba\),
\(\alpha\co\id_{\E}\Rightarrow h\). Suppose that \(h\neq\id_{\E}\),
and choose an object \(e\) of minimal coheight such that either
\(h(e)\neq e\) or \(\alpha_e\neq1_e\). Then
\(\alpha_e\co e\to h(e)\) is nonidentity. For every nonidentity morphism
\(g\co e\to y\), the object \(y\) has smaller coheight, so
\(h(y)=y\) and \(\alpha_y=1_y\). Naturality gives
\(h(g)\alpha_e=g\). If \(k\co h(e)\to y\) is another factorization with
\(k\alpha_e=g\), applying \(h\) and using
\(h(\alpha_e)=1_{h(e)}\) gives \(k=h(g)\). Hence \(\alpha_e\) is initial
in \(\widehat{e\downarrow\E}\). Since \(\alpha_e\) is vertical, \(e\)
is an up beat object of \(p\), contradicting minimality.

Finally, suppose that \(h\) is connected to \(\id_{\E}\) by a zigzag
of functors and natural transformations over \(\Ba\),
\[
h=h_0\Longleftrightarrow h_1\Longleftrightarrow\cdots
\Longleftrightarrow h_n=\id_{\E}.
\]
The last natural transformation joins \(h_{n-1}\) to \(\id_{\E}\), in
one of the two possible directions. By the preceding arguments,
\(h_{n-1}=\id_{\E}\). Repeating the argument backwards along the zigzag
gives
\[
h_{n-2}=\id_{\E},\ldots,h_0=\id_{\E}.
\]
Therefore \(h=\id_{\E}\).
\end{proof}

\begin{corollary}
Any two relative cores of \(p\) are isomorphic in \(\cat/\Ba\).
\end{corollary}

\begin{proof}
Let \(p_1\co\E_1\to\Ba\) and \(p_2\co\E_2\to\Ba\) be relative cores
of \(p\), with inclusions \(\inc_k\) and retractions \(r_k\). The composites
\(r_2\inc_1\co\E_1\to\E_2\) and
\(r_1\inc_2\co\E_2\to\E_1\) lie over \(\Ba\). Their two composites are
strongly homotopic over \(\Ba\) to the corresponding identity functors.
Lemma~\ref{lem:relative-rigidity} implies that both composites are
identities. Hence the two relative cores are isomorphic over \(\Ba\).
\end{proof}

\subsection{Minimal bifibrations as fiber bundles}

The following result is the categorical counterpart of
\cite[Theorem 6.9]{cianci2019fibrationsfinitetopologicalspaces}.

\begin{proposition}
\label{prop:minimal-bifibration-bundle}
Let \(p\co\E\to\Ba\) be a Grothendieck bifibration over a connected
category \(\Ba\). Suppose that every fiber \(\E_b\) is a minimal finite
acyclic category. Then, after choosing an opcleavage, the associated
transport defines a fiber bundle
\[
P_p\co\Ba\longrightarrow\cat,
\qquad
P_p(b)=\E_b,
\qquad
P_p(f)=f_!,
\]
and \(p\) is isomorphic over \(\Ba\) to the projection
\[
\pi_{P_p}\co\int^{\Ba}P_p\longrightarrow\Ba.
\]
\end{proposition}

\begin{proof}
Choose a cleavage and an opcleavage for \(p\). For every morphism
\(f\co b\to b'\), they determine transport functors
\[
f_!\co\E_b\longrightarrow\E_{b'},
\qquad
f^*\co\E_{b'}\longrightarrow\E_b,
\]
with an adjunction \(f_!\dashv f^*\). The unit and counit give strong
homotopies
\[
\id_{\E_b}\simeq f^*f_!,
\qquad
f_!f^*\simeq\id_{\E_{b'}}.
\]
Since the fibers are minimal finite acyclic categories,
Lemma~\ref{lem:rigidity} gives
\[
f^*f_!=\id_{\E_b},
\qquad
f_!f^*=\id_{\E_{b'}}.
\]
Hence \(f_!\) and \(f^*\) are mutually inverse isomorphisms.

The chosen opcleavage provides canonical coherence isomorphisms
\[
(gf)_!\cong g_!f_!,
\qquad
(1_b)_!\cong\id_{\E_b}.
\]
Every isomorphism in an acyclic category is an identity. Consequently, every natural isomorphism between functors with codomain an
acyclic category is an identity natural transformation. The coherence isomorphisms are therefore
identities, and
\[
(gf)_!=g_!f_!,
\qquad
(1_b)_!=\id_{\E_b}.
\]
Moreover, the vertical comparison isomorphisms between the chosen lift of
a composite and the composite of the chosen lifts are identities. Hence the
chosen opcleavage is closed.
Thus the assignments
\[
P_p(b)=\E_b,
\qquad
P_p(f)=f_!
\]
define a strict functor \(P_p\co\Ba\to\cat\).

Every \(P_p(f)=f_!\) is an isomorphism. Since \(\Ba\) is connected, all
fibers are isomorphic to any fixed fiber \(\E_{b_0}\). Hence \(P_p\) is a
fiber bundle over \(\Ba\).

It remains to compare its Grothendieck construction with \(\E\). Let
\(
\varphi_{f,e}\co e\longrightarrow f_!e
\)
denote the chosen opcartesian lift of \(f\co b\to b'\) from
\(e\in\E_b\). Define
\(
\Phi\co\int^{\Ba}P_p\longrightarrow\E
\)
on objects by \(\Phi(b,e)=e\), and on morphisms by
\(
\Phi(f,\alpha)=\alpha\circ\varphi_{f,e},
\)
where
\(\alpha\co f_!e\to e'\) is vertical.

By the universal property of \(\varphi_{f,e}\), every morphism
\(u\co e\to e'\) of \(\E\) over \(f\) admits a unique factorization
\(
u=\alpha\circ\varphi_{f,e}
\)
with \(\alpha\) vertical. Therefore \(\Phi\) is bijective on objects and
morphisms. The identities and composition in
\(\int^{\Ba}P_p\) correspond to those of \(\E\), because the chosen
opcleavage is closed. Hence \(\Phi\) is an isomorphism of categories.
Moreover, it commutes with the projections to \(\Ba\). Thus
\(
p\cong\pi_{P_p}
\)
in the slice category \(\cat/\Ba\).
\end{proof}
\section{Monodromy and classification of fiber bundles}\label{sec:Monodromy_categorical_fiber_bundles}

\subsection{Fundamental groupoid and monodromy}\label{sec:fundgroup}

\subsubsection{Localization and fundamental group}

For a category \(\CC\), let
\(
\Pi_1(\CC):=\CC[\operatorname{Mor}(\CC)^{-1}]
\)
denote its localization at all morphisms. This is the \emph{fundamental
groupoid} of \(\CC\). Its morphisms are represented by zigzags in \(\CC\),
modulo the usual localization relations.

The canonical functor
\(
\ell_{\CC}\co\CC\longrightarrow\Pi_1(\CC)
\)
is universal among functors sending every morphism of \(\CC\) to an
isomorphism. Thus, if \(F\co\CC\to\D\) sends every morphism to an
isomorphism, there is a unique functor
\(
\overline{F}\co\Pi_1(\CC)\longrightarrow\D
\)
such that \(F=\overline{F}\circ\ell_{\CC}\). We shall usually denote the
induced functor again by \(F\).

The universal property also applies to natural transformations.

\begin{lemma}\label{lemma:universal_localization}
Let \(F,G\co\CC\to\D\) send every morphism of \(\CC\) to an isomorphism.
Restriction along \(\ell_{\CC}\) induces a bijection
\[
\operatorname{Nat}_{\Pi_1(\CC)}(F,G)
\cong
\operatorname{Nat}_{\CC}(F,G).
\]
\end{lemma}

\begin{proof}
Only surjectivity requires verification. Let
\(\alpha\co F\Rightarrow G\) be natural on \(\CC\). For every
\(f\co c\to d\), naturality gives
\(
\alpha_d\circ F(f)=G(f)\circ\alpha_c.
\)
Since \(F(f)\) and \(G(f)\) are isomorphisms, this is equivalent to
\(
\alpha_c\circ F(f)^{-1}
 =
G(f)^{-1}\circ\alpha_d.
\)

Consequently, the family \((\alpha_c)_{c\in\CC}\) is natural with respect
to both the images of the morphisms of \(\CC\) and their formal inverses.
It is therefore natural with respect to every morphism of
\(\Pi_1(\CC)\). Uniqueness is immediate because the localization does not
change the objects.
\end{proof}

If \(\CC\) is connected and \(c_0\in\CC\), we write
\(
\pi_1(\CC,c_0)
 =
\Auto_{\Pi_1(\CC)}(c_0).
\)
The inclusion of the full subgroupoid on \(c_0\) is an equivalence
\(
\B\pi_1(\CC,c_0)\simeq\Pi_1(\CC),
\)
where a group is regarded as a one-object groupoid.

\subsubsection{A presentation of the fundamental group}

We recall the presentation of \(\pi_1(\CC,c_0)\) given in
\cite[Definition 11.3.5 and Proposition 11.3.6]{Richter_2020}.

A set \(X\) of morphisms of \(\CC\) is a \emph{tree} if the graph
obtained by forgetting the orientations of its morphisms is a tree. If
\(\CC\) is connected, it is a \emph{spanning tree} if it contains every
object as a vertex.

\begin{proposition}[{\cite[Proposition 11.3.6]{Richter_2020}}]
Let \(\CC\) be connected, let \(c_0\in\CC\), and let \(X\) be a spanning
tree. Then \(\pi_1(\CC,c_0)\) admits a presentation with one generator
\([f]\) for every morphism \(f\) of \(\CC\), subject to the relations
\[
[f]=1 \quad\text{if } f\in X,
\qquad
[1_c]=1,
\qquad
[g\circ f]=[g][f]
\]
whenever \(f\) and \(g\) are composable.
\end{proposition}

If \(\CC\) is presented by a quiver with relations, one may choose a
spanning tree in the underlying graph. The remaining edges provide
generators, while the relations of the category induce relations in the
fundamental group.

\begin{example}\label{ex:fundamental-S}
For the category \(S\) of Example~\ref{ex:S}, choose \(f\) as a spanning
tree. The remaining arrow \(g\) determines one generator and there are no
relations. Hence
\[
\pi_1(S,x)\cong\mathbb{Z}.
\]
\end{example}

\begin{example}\label{ex:fundamental-three-arrows}
Let \(\Ba\) be the category from
Example~\ref{Ex:Circle_Twist_S_3}. Choosing \(f_1\) as a spanning tree
leaves the generators determined by \(f_2\) and \(f_3\), with no
relations. Therefore
\[
\pi_1(\Ba,x)\cong\mathbb{Z}*\mathbb{Z}.
\]
\end{example}

\begin{example}\label{ex:fundamental-projective-plane}
For the projective plane category \(\Pa\) of Example~\ref{ex:P}, choose
\(f_1\) and \(g_1\) as a spanning tree. The defining relations become
\[
1=[g_2][f_2],
\qquad
[g_2]=[f_2].
\]
Thus
\[
\pi_1(\Pa,X)\cong\mathbb{Z}/2\mathbb{Z}.
\]
\end{example}

\begin{example}\label{ex:fundamental-klein-bottle}
Consider the total category \(\int^S K\) from
Example~\ref{ex:klein bottle}. Choose a spanning tree containing the
vertical arrows and the horizontal arrows lying over \(f\). If
\[
a=[(1_x,g)]=[(1_y,g)],
\qquad
b=[(g,1_y)],
\qquad
c=[(g,1_x)],
\]
the relations in the Grothendieck construction reduce to
\(
b=ac,
\text{ and }
ba=c.
\)
Eliminating \(c\) gives
\(
b=aba.
\)
Equivalently,
\[
\pi_1\left(\int^S K,(x,x)\right)
\cong
\langle a,b\mid b=aba\rangle,
\]
which is a standard presentation of the fundamental group of the Klein
bottle.
\end{example}

\subsubsection{Normalization of fiber bundles}\label{sec:trivia}

Every fiber bundle can be normalized in two steps: first by identifying all
its fibers with a fixed category, and then by making its transition
automorphisms trivial on a spanning tree.

\begin{proposition}[Constant fiber reduction]
\label{prop:constant-fiber-reduction}
Let \(P\co\Ba\to\cat\) be a fiber bundle with fiber \(\F\). Then \(P\) is
naturally isomorphic to a fiber bundle
\[
P^{\mathrm{c}}\co\Ba\longrightarrow\cat
\]
such that \(P^{\mathrm{c}}(b)=\F\) for every \(b\in\Ba\).
\end{proposition}

\begin{proof}
Choose isomorphisms
\(
\psi_b\co P(b)\longrightarrow\F
\)
and define
\[
P^{\mathrm{c}}(f)
 =
\psi_{b'}\circ P(f)\circ\psi_b^{-1}
\]
for every \(f\co b\to b'\). The family \((\psi_b)_{b\in\Ba}\) defines a
natural isomorphism \(P\Rightarrow P^{\mathrm{c}}\).
\end{proof}

We may therefore assume that \(P(b)=\F\) for every \(b\), so that every
transition functor belongs to \(\Auto(\F)\).

\begin{theorem}[Spanning tree normalization]
\label{teo:spanning_tree}
Let \(P\co\Ba\to\cat\) be a fiber bundle with constant fiber \(\F\), and
let \(X\) be a spanning tree of \(\Ba\). Then \(P\) is naturally isomorphic
to a fiber bundle
\(
P^X\co\Ba\longrightarrow\cat
\)
such that
\(
P^X(f)=\id_{\F}
\)
for every edge \(f\in X\).
\end{theorem}

\begin{proof}
Fix \(b_0\in\Ba\). For every \(b\in\Ba\), let
\(
\gamma_b\co b_0\longrightarrow b
\)
be the unique path in \(X\), regarded as a morphism of \(\Pi_1(\Ba)\).
Since \(P\) sends every morphism to an isomorphism, it factors through
\(\Pi_1(\Ba)\), and hence \(P(\gamma_b)\in\Auto(\F)\) is defined.

Set \(P^X(b)=\F\) and, for \(f\co b\to b'\), define
\[
P^X(f)
 =
P(\gamma_{b'})^{-1}\circ P(f)\circ P(\gamma_b).
\]
Functoriality follows immediately from that of \(P\). If \(f\) is an edge
of \(X\), then either
\[
\gamma_{b'}=f\circ\gamma_b
\qquad\text{or}\qquad
\gamma_b=f^{-1}\circ\gamma_{b'}
\]
in \(\Pi_1(\Ba)\), according to the orientation of \(f\). In either case,
\(
P^X(f)=\id_{\F}.
\)

Finally, the family
\(
\theta_b=P(\gamma_b)\co P^X(b)\longrightarrow P(b)
\)
defines a natural isomorphism
\(
\theta\co P^X\Rightarrow P.
\)
\end{proof}

The normalization \(P^X\) is completely determined by the transition
automorphisms associated with the edges outside \(X\), subject to the
relations of \(\Ba\). Equivalently, it is determined by its monodromy
representation.

\subsection{Monodromy and strict gauge transformations}\label{sec:gauge}

\subsubsection{Strict gauge group}
Let \(P\co\Ba\to\cat\) be a fiber bundle.

\begin{definition}
The \emph{gauge group} of \(P\) is
\[
\operatorname{Gau}(P)
 =
\left\{
g\in\Auto\left(\int^{\Ba}P\right)
\ \middle|\
\pi_P\circ g=\pi_P
\right\}.
\]
\end{definition}

The Grothendieck construction comes with canonical opcartesian lifts
\[
(f,1_{P(f)(x)})
 \co
(b,x)\longrightarrow(b',P(f)(x)).
\]

\begin{definition}
The \emph{strict gauge group} \(\operatorname{Gau}_{\mathrm{str}}(P)\) is
the subgroup of \(\operatorname{Gau}(P)\) consisting of the automorphisms
that preserve these canonical opcartesian lifts.
\end{definition}

Explicitly, if \(g(b,x)=(b,g_b(x))\), the strictness condition is
\[
g(f,1_{P(f)(x)})
 =
(f,1_{P(f)(g_b(x))}).
\]

\begin{proposition}\label{prop:strict-gauge-natural}
There is a natural group isomorphism
\[
\operatorname{Gau}_{\mathrm{str}}(P)
\cong
\operatorname{Nat}_{\mathrm{Iso}}(P,P).
\]
\end{proposition}

\begin{proof}
Let \(g\in\operatorname{Gau}_{\mathrm{str}}(P)\). Since \(g\) lies over
\(\id_{\Ba}\), its restriction to the fiber over \(b\) defines an
automorphism
\(
g_b\co P(b)\longrightarrow P(b).
\)
Preservation of the canonical opcartesian lifts implies
\(
P(f)\circ g_b=g_{b'}\circ P(f)
\)
for every \(f\co b\to b'\). Thus, \((g_b)_{b\in\Ba}\) defines a natural
automorphism of \(P\).

Conversely, a natural automorphism
\(\eta\co P\Rightarrow P\) determines an automorphism over \(\Ba\) by
\[
(b,x)\longmapsto(b,\eta_b(x)),
\qquad
(f,\alpha)\longmapsto(f,\eta_{b'}(\alpha)).
\]
Naturality ensures that this is well defined, and the resulting
automorphism preserves the canonical opcartesian lifts. These constructions
are mutually inverse and compatible with composition.
\end{proof}

In contrast, an arbitrary element of \(\operatorname{Gau}(P)\) need not
preserve the canonical opcartesian lifts and therefore need not arise from
a natural automorphism of \(P\).

\subsubsection{Monodromy}

Fix \(b_0\in\Ba\). Since \(P\) sends every morphism to an isomorphism, it
induces a functor
\(
\overline{P}\co\Pi_1(\Ba)\longrightarrow\cat.
\)
After identifying \(P(b_0)\) with \(\F\), restriction to the automorphism
group of \(b_0\) gives the \emph{monodromy representation}
\(
\rho_P\co\pi_1(\Ba,b_0)\longrightarrow\Auto(\F).
\)
Its image
\(
M_P:=\Ima(\rho_P)\leq\Auto(\F)
\)
is the \emph{monodromy subgroup} of \(P\).

Changing the identification \(P(b_0)\cong\F\) conjugates \(\rho_P\).
Consequently, its conjugacy class and the conjugacy class of \(M_P\) depend
only on the isomorphism class of the bundle.

\begin{proposition}\label{prop:gauge-centralizer}
Let \(P\co\Ba\to\cat\) be a fiber bundle with connected base and fiber
\(\F\). Then
\[
\operatorname{Gau}_{\mathrm{str}}(P)
\cong
C_{\Auto(\F)}(M_P).
\]
\end{proposition}

\begin{proof}
By Proposition~\ref{prop:strict-gauge-natural} and
Lemma~\ref{lemma:universal_localization},
\[
\operatorname{Gau}_{\mathrm{str}}(P)
\cong
\operatorname{Nat}_{\mathrm{Iso}}(\overline{P},\overline{P}).
\]
Since \(\Pi_1(\Ba)\) is connected, a natural automorphism of
\(\overline{P}\) is determined by its component
\(u\in\Auto(\F)\) at \(b_0\). Naturality with respect to every loop
\(\gamma\in\pi_1(\Ba,b_0)\) is precisely the condition
\[
u\rho_P(\gamma)=\rho_P(\gamma)u.
\]
Thus \(u\) belongs to \(C_{\Auto(\F)}(M_P)\), and every element of this
centralizer extends uniquely to a natural automorphism of
\(\overline{P}\).
\end{proof}

\begin{remark}
If \(\theta\co P\Rightarrow Q\) is a natural isomorphism, conjugation
induces the group isomorphism
\[
\operatorname{Nat}_{\mathrm{Iso}}(P,P)
\longrightarrow
\operatorname{Nat}_{\mathrm{Iso}}(Q,Q),
\qquad
\eta\longmapsto\theta\circ\eta\circ\theta^{-1}.
\]
Therefore the strict gauge group is unchanged by the normalizations of
Section~\ref{sec:trivia}.
\end{remark}

\begin{example}
For the bundle \(K\co S\to\cat\), the monodromy representation is the
surjection
\[
\mathbb{Z}\cong\pi_1(S,x)
\longrightarrow
\Auto(S)\cong\mathbb{Z}_2
\]
sending a generator to the automorphism that exchanges \(f\) and \(g\).
Since \(\mathbb{Z}_2\) is abelian,
\[
\operatorname{Gau}_{\mathrm{str}}(K)
\cong
C_{\mathbb{Z}_2}(\mathbb{Z}_2)
\cong
\mathbb{Z}_2.
\]
\end{example}

\begin{example}
For the bundle of Example~\ref{Ex:Circle_Twist_S_3}, the monodromy subgroup
is generated by \((12)\) and \((23)\), and hence equals \(S_3\). Since the
center of \(S_3\) is trivial,
\[
\operatorname{Gau}_{\mathrm{str}}(P)
\cong
C_{S_3}(S_3)
 =
1.
\]
\end{example}

\subsection{Classification of fiber bundles}\label{sec:classification}

Fix a connected category \(\Ba\), a category \(\F\), and let
\(
G=\Auto(\F).
\)
We regard \(G\) as a one-object groupoid \(\B G\).

By Proposition~\ref{prop:constant-fiber-reduction}, every bundle with fiber
\(\F\) is isomorphic to a functor \(P\co\Ba\to\B G\), where
\(G=\Auto(\F)\). Equivalently, it is a map
\(\varphi_P\co\operatorname{Mor}(\Ba)\to G\) satisfying
\(\varphi_P(1_b)=1\) and
\(\varphi_P(gf)=\varphi_P(g)\varphi_P(f)\). Two such cocycles are
cohomologous if
\(\psi(f)=h(b')\varphi(f)h(b)^{-1}\) for some
\(h\co\operatorname{Ob}(\Ba)\to G\). We write
\(H^1(\Ba;G)\) for the resulting pointed set.

\begin{theorem}\label{thm:classification-H1}
There are natural bijections
\[
\left\{
\begin{array}{c}
\text{isomorphism classes of fiber bundles}\\
\text{over \(\Ba\) with fiber \(\F\)}
\end{array}
\right\}
\cong
H^1(\Ba;G)
\cong
\operatorname{Hom}\bigl(\pi_1(\Ba,b_0),G\bigr)/G,
\]
where \(G\) acts on the set of homomorphisms by conjugation.
\end{theorem}

\begin{proof}
The first bijection follows from the preceding description of bundles as
\(G\)-valued cocycles and bundle isomorphisms as coboundary transformations.

For the second, every functor
\(
P\co\Ba\longrightarrow\B G
\)
factors uniquely through \(\Pi_1(\Ba)\). Since \(\Ba\) is connected,
choosing \(b_0\) identifies functors
\(\Pi_1(\Ba)\to\B G\), up to natural isomorphism, with homomorphisms
\(
\rho\co\pi_1(\Ba,b_0)\longrightarrow G
\)
up to conjugation in \(G\). Under this correspondence, \(P\) is sent to
its monodromy representation \(\rho_P\).
\end{proof}

Equivalently, after choosing a spanning tree \(X\), every bundle is
isomorphic to one that is trivial on \(X\), and its remaining transition
automorphisms determine the homomorphism
\(
\rho_P\co\pi_1(\Ba,b_0)\longrightarrow G.
\)

\begin{remark}
If \(G\) is abelian, conjugation is trivial and therefore
\[
H^1(\Ba;G)
\cong
\operatorname{Hom}\bigl(\pi_1(\Ba,b_0),G\bigr).
\]
This agrees with the first cohomology group of the nerve of \(\Ba\) with
constant coefficients in \(G\); see
\cite{WebbCohomology}.
\end{remark}

\subsection{Sections and fixed points}\label{sec:sections}

Let \(P\co\Ba\to\cat\) be a fiber bundle with connected base and fiber
\(\F\). Fix \(b_0\in\Ba\), identify \(P(b_0)\) with \(\F\), and let
\(\rho_P\co\pi_1(\Ba,b_0)\to\Auto(\F)\) be the monodromy
representation. We denote by \(\mathbf{1}\co\Ba\to\cat\) the constant
functor with value the terminal category.

A \emph{horizontal section} of
\(\pi_P\co\int^{\Ba}P\to\Ba\) is a section that preserves the canonical
opcartesian lifts. Thus, if \(s(b)=(b,x_b)\), then
\(P(f)(x_b)=x_{b'}\) and \(s(f)=(f,1_{x_{b'}})\) for every
\(f\co b\to b'\). An object \(x\in\F\) is a \emph{strict fixed point}
of the monodromy if \(\rho_P(\gamma)(x)=x\) for every
\(\gamma\in\pi_1(\Ba,b_0)\).

\begin{proposition}\label{prop:sections-fixed-points}
There are natural bijections between the following sets:
\begin{enumerate}
\item strict fixed points of the monodromy;
\item horizontal sections of \(\pi_P\);
\item natural transformations \(\mathbf{1}\Rightarrow P\).
\end{enumerate}
Moreover, arbitrary sections of \(\pi_P\) are naturally identified with
lax natural transformations \(\mathbf{1}\Rightarrow P\).
\end{proposition}

\begin{proof}
A natural transformation \(\eta\co\mathbf{1}\Rightarrow P\) assigns an
object \(x_b=\eta_b(\ast)\in P(b)\) to each \(b\in\Ba\), subject to
\(P(f)(x_b)=x_{b'}\) for every \(f\co b\to b'\). This is precisely the
data of a horizontal section, defined by \(s(b)=(b,x_b)\) and
\(s(f)=(f,1_{x_{b'}})\).

Let \(\overline P\co\Pi_1(\Ba)\to\cat\) be the extension of \(P\) to
the fundamental groupoid. If \(x\in\F\) is fixed by the monodromy, choose
for every \(b\in\Ba\) a morphism \(\gamma_b\co b_0\to b\) in
\(\Pi_1(\Ba)\) and set \(x_b=\overline P(\gamma_b)(x)\). This is
independent of the choice of \(\gamma_b\), since two such choices differ
by a loop at \(b_0\), and it satisfies \(P(f)(x_b)=x_{b'}\). Conversely,
a horizontal section determines the fixed point \(x_{b_0}\).

Finally, an arbitrary section has the form \(s(b)=(b,x_b)\) and
\(s(f)=(f,\alpha_f)\), where
\(\alpha_f\co P(f)(x_b)\to x_{b'}\), with
\(\alpha_{1_b}=1_{x_b}\) and
\(\alpha_{gf}=\alpha_g\circ P(g)(\alpha_f)\). These are exactly the data
and axioms of a lax natural transformation
\(\mathbf{1}\Rightarrow P\).
\end{proof}

Thus every horizontal section is a section, but the converse need not hold:
a fiber bundle may admit a section even when its monodromy has no strict
fixed points.

\begin{corollary}\label{cor:covering-sections}
Let \(P\co\Ba\to\Sets\) be a fiber bundle with discrete fiber \(F\).
Then \(\pi_P\) admits a section if and only if the monodromy action on
\(F\) has a fixed point. In this case every section is horizontal.
\end{corollary}

\begin{proof}
Since \(F\) is discrete, a morphism
\(\alpha_f\co P(f)(x_b)\to x_{b'}\) exists if and only if
\(P(f)(x_b)=x_{b'}\), in which case it is an identity. Hence every
section is horizontal, and the result follows from
Proposition~\ref{prop:sections-fixed-points}.
\end{proof}

For the covering of Example~\ref{ex:double_cover}, the monodromy exchanges
the two elements of the fiber. It has no fixed point, so the covering has no
section.

\section{Fundamental-groupoid reductions and classification}
\label{sec:classification-reductions}

The classification of fiber bundles over a connected category \(\Ba\) with
fixed fiber \(\F\) depends on the fundamental group of the base. We therefore
introduce an object-removal procedure that preserves the fundamental
groupoid. This reduction is generally weaker than strong homotopy reduction
and is designed to preserve the classification of fiber bundles rather than
the full strong homotopy type.

We begin with a categorical version of the one-dimensional part of
Quillen's Theorem~A \cite{Quillen1973}.

\begin{theorem}[Fundamental-groupoid version of Quillen's Theorem~A]
\label{thm:quillen-pi1}
Let \(F\co\CC\to\D\) be a functor. Suppose that
\(\Pi_1(d\downarrow F)\) is equivalent to the terminal groupoid for every
object \(d\in\D\). Then the induced functor
\(\Pi_1(F)\co\Pi_1(\CC)\to\Pi_1(\D)\) is an equivalence of groupoids.
Consequently, for every \(c\in\CC\), the induced homomorphism
\(\pi_1(\CC,c)\to\pi_1(\D,F(c))\) is an isomorphism.
\end{theorem}

\begin{proof}
For every \(d\in\D\), choose an object
\((c_d,\alpha_d)\in d\downarrow F\), where
\(\alpha_d\co d\to F(c_d)\). Let
\(q_d\co d\downarrow F\to\CC\) denote the canonical projection.

We first construct a functor \(G\co\D\to\Pi_1(\CC)\). On objects, set
\(G(d)=c_d\). Let \(u\co d\to d'\) be a morphism of \(\D\). The pairs
\((c_d,\alpha_d)\) and
\((c_{d'},\alpha_{d'}\circ u)\) are objects of \(d\downarrow F\).
Since \(\Pi_1(d\downarrow F)\) is equivalent to the terminal groupoid,
there is a unique morphism
\(\xi_u\co(c_d,\alpha_d)\to(c_{d'},\alpha_{d'}\circ u)\) in
\(\Pi_1(d\downarrow F)\). Define
\(G(u)=\Pi_1(q_d)(\xi_u)\).

For \(u=1_d\), uniqueness gives
\(\xi_{1_d}=1_{(c_d,\alpha_d)}\), and hence \(G(1_d)=1_{c_d}\).

Let \(u\co d\to d'\) and \(v\co d'\to d''\) be composable.
Precomposition with \(u\) defines a functor
\(u^*\co d'\downarrow F\to d\downarrow F\), given by
\(u^*(c,\alpha)=(c,\alpha\circ u)\). The morphisms
\(u^*(\xi_v)\circ\xi_u\) and \(\xi_{vu}\) have the same source and target
in \(\Pi_1(d\downarrow F)\), namely \((c_d,\alpha_d)\) and
\((c_{d''},\alpha_{d''}\circ v\circ u)\). By uniqueness,
\(u^*(\xi_v)\circ\xi_u=\xi_{vu}\). Applying \(\Pi_1(q_d)\), and using
\(q_d\circ u^*=q_{d'}\), gives \(G(vu)=G(v)G(u)\). Thus \(G\) is a
functor.

Since \(\Pi_1(\CC)\) is a groupoid, \(G\) factors uniquely through the
localization of \(\D\), giving a functor
\(\overline G\co\Pi_1(\D)\to\Pi_1(\CC)\).

By the definition of \(\xi_u\), the equality
\(\Pi_1(F)(G(u))\circ\alpha_d=\alpha_{d'}\circ u\) holds in
\(\Pi_1(\D)\). Hence the family \((\alpha_d)_{d\in\D}\) defines a
natural isomorphism
\(\alpha\co\id_{\Pi_1(\D)}\Rightarrow\Pi_1(F)\circ\overline G\).

It remains to construct a natural isomorphism in the other direction. For
every \(c\in\CC\), the pairs \((c,1_{F(c)})\) and
\((c_{F(c)},\alpha_{F(c)})\) are objects of \(F(c)\downarrow F\).
Let
\(\zeta_c\co(c,1_{F(c)})\to(c_{F(c)},\alpha_{F(c)})\)
be the unique morphism between them in
\(\Pi_1(F(c)\downarrow F)\), and define
\(\beta_c=\Pi_1(q_{F(c)})(\zeta_c)\). Thus
\(\beta_c\co c\to\overline G(F(c))\) is an isomorphism in
\(\Pi_1(\CC)\).

Let \(h\co c\to c'\) be a morphism of \(\CC\). It determines a morphism
\[
\widehat h\co
(c,1_{F(c)})
\longrightarrow
(c',F(h))
\]
in \(F(c)\downarrow F\). Precomposition with \(F(h)\) defines a functor
\(F(h)^*\co F(c')\downarrow F\to F(c)\downarrow F\). In
\(\Pi_1(F(c)\downarrow F)\), the morphisms
\(F(h)^*(\zeta_{c'})\circ\widehat h\) and
\(\xi_{F(h)}\circ\zeta_c\) have the same source and target, namely
\((c,1_{F(c)})\) and
\((c_{F(c')},\alpha_{F(c')}\circ F(h))\). By uniqueness, they coincide.

Applying the projection to \(\Pi_1(\CC)\) gives
\(\beta_{c'}\circ h
=\overline G(\Pi_1(F)(h))\circ\beta_c\).
Therefore \((\beta_c)_{c\in\CC}\) defines a natural isomorphism
\(\beta\co\id_{\Pi_1(\CC)}
\Rightarrow\overline G\circ\Pi_1(F)\).

Thus \(\overline G\) is a quasi-inverse of \(\Pi_1(F)\), and
\(\Pi_1(F)\) is an equivalence of groupoids.
\end{proof}

The dual statement holds if \(\Pi_1(F\downarrow d)\) is equivalent to the
terminal groupoid for every object \(d\in\D\).

\subsection{Negligible objects}

\begin{definition}
Let \(\A\) be a finite acyclic category, let \(x\in\A\), and set
\(\A_x=\A\setminus\{x\}\), with inclusion
\(\inc_x\co\A_x\to\A\).

The object \(x\) is \emph{upper fundamental-groupoid negligible} if
\(\Pi_1(x\downarrow\inc_x)\) is equivalent to the terminal groupoid. It
is \emph{lower fundamental-groupoid negligible} if
\(\Pi_1(\inc_x\downarrow x)\) is equivalent to the terminal groupoid.
We call \(x\) \emph{fundamental-groupoid negligible} if it satisfies
either condition.
\end{definition}

Since \(\inc_x\) is the inclusion of a full subcategory, there are
canonical identifications
\(x\downarrow\inc_x\cong\widehat{x\downarrow\A}\) and
\(\inc_x\downarrow x\cong\widehat{\A\downarrow x}\). Thus the definition
can be checked directly from the punctured comma categories used to define
beat objects.

\begin{theorem}[Negligible-object removal]
\label{thm:pi1-negligible-removal}
If \(x\) is a fundamental-groupoid negligible object of a finite acyclic
category \(\A\), then
\(\Pi_1(\inc_x)\co\Pi_1(\A_x)\to\Pi_1(\A)\) is an equivalence of
groupoids.
\end{theorem}

\begin{proof}
Suppose first that \(x\) is upper fundamental-groupoid negligible. For
every \(a\in\A_x\), the comma category \(a\downarrow\inc_x\) has the
initial object \((a,1_a)\), and hence its fundamental groupoid is terminal.
By hypothesis, the same is true for \(x\downarrow\inc_x\). The result
follows from Theorem~\ref{thm:quillen-pi1}.

If \(x\) is lower fundamental-groupoid negligible, apply the dual version
of Theorem~\ref{thm:quillen-pi1} to the comma categories
\(\inc_x\downarrow a\).
\end{proof}

\begin{proposition}
Every beat object is fundamental-groupoid negligible.
\end{proposition}

\begin{proof}
If \(x\) is an up beat object, then
\(\widehat{x\downarrow\A}\cong x\downarrow\inc_x\) has an initial
object, so its fundamental groupoid is terminal. Thus \(x\) is upper
fundamental-groupoid negligible. The down beat case is dual.
\end{proof}

The converse need not hold. A fundamental-groupoid negligible object only requires the fundamental groupoid of the relevant punctured comma category to be terminal; the comma category need not have an initial or terminal object.
Consequently, its removal preserves the fundamental groupoid but need not
preserve the strong homotopy type.

\begin{corollary}[Classification after reducing the base]
\label{cor:classification-pi1-negligible}
Let \(\Ba'\) be obtained from a connected finite acyclic category \(\Ba\)
by successively removing fundamental-groupoid negligible objects. For every
category \(\F\), the inclusion \(\Ba'\to\Ba\) induces a bijection between
isomorphism classes of fiber bundles over \(\Ba\) with fiber \(\F\) and
isomorphism classes of fiber bundles over \(\Ba'\) with fiber \(\F\).
\end{corollary}

\begin{proof}
Successive applications of
Theorem~\ref{thm:pi1-negligible-removal} give an equivalence
\(\Pi_1(\Ba')\simeq\Pi_1(\Ba)\). Under the classification theorem, the
induced bijection is obtained by transporting monodromy representations
along this equivalence. Hence the two bases determine the same isomorphism
classes of fiber bundles with fiber \(\F\).
\end{proof}

Thus fundamental-groupoid negligible objects define reductions of the base
that are generally weaker than beat-object reductions but preserve all the
information required to classify fiber bundles with a fixed fiber. Beat
objects are adapted to strong homotopy reduction, whereas
fundamental-groupoid negligible objects are adapted specifically to
monodromy and classification.

\section{Detailed examples}\label{sec:examples}

\subsection{The Klein bottle, second model (fiber \(S'\))}
Take the category \(\mathcal{S}\) defined in Example~\ref{ex:S} and the functor \(K' : \mathcal{S} \to \mathbf{cat}\) defined by \(K'(x)=\mathcal{S}'\) where \(\mathcal{S}'\) is the poset from Example~\ref{ex:Sprime} where we will see it with the following diagram
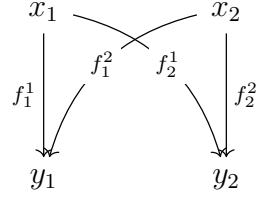
\[
% https://tikzcd.yichuanshen.de/#N4Igdg9gJgpgziAXAbVABwnAlgFyxMJZABgBpiBdUkANwEMAbAVxiRAA8B9ARhAF9S6TLnyEU3clVqMWbLgCZ+gkBmx4CRMtyn1mrRCACePJULWiiE7dV2yDxxXykwoAc3hFQAMwBOEALZIZCA4EEjyNjL6IF48AHq81Ax0AEYwDAAKwupiID5YrgAWOKYxfoGIwaFIAMyRemyx8gkgSanpWeYaBlhg2LCtIGlgULXEAt7lSBIhYYgR0g0Gsdxxim1pmdkWPX1YA9TDo4gAtDXjyr4B09TViHWLdjGczesgyZudIt3vMF4lTj4QA
\begin{tikzcd}[column sep=4em, row sep=4em]
x_1 \arrow[d, "f_1^1"'] \arrow[rd, "f_2^1" description, bend left] & x_2 \arrow[ld, "f_1^2" description, bend right] \arrow[d, "f_2^2"] \\
y_1                                                                & y_2                                                               
\end{tikzcd}
\]
and \(K'(f)=1_{\mathcal{S}'}\) and \(K'(g)=s'\), where \(s'\) swaps the subscripts: \(s'(x_i)=x_j\) and \(s'(y_i)=y_j\) with \(i\neq j\).

For the total category $\int K'$ the set of generating morphisms consists of the vertical arrows coming from the poset \(S'\) and the horizontal lifts of \(f\) and \(g\) twisted by the swap \(\sigma\). The complete diagram is:
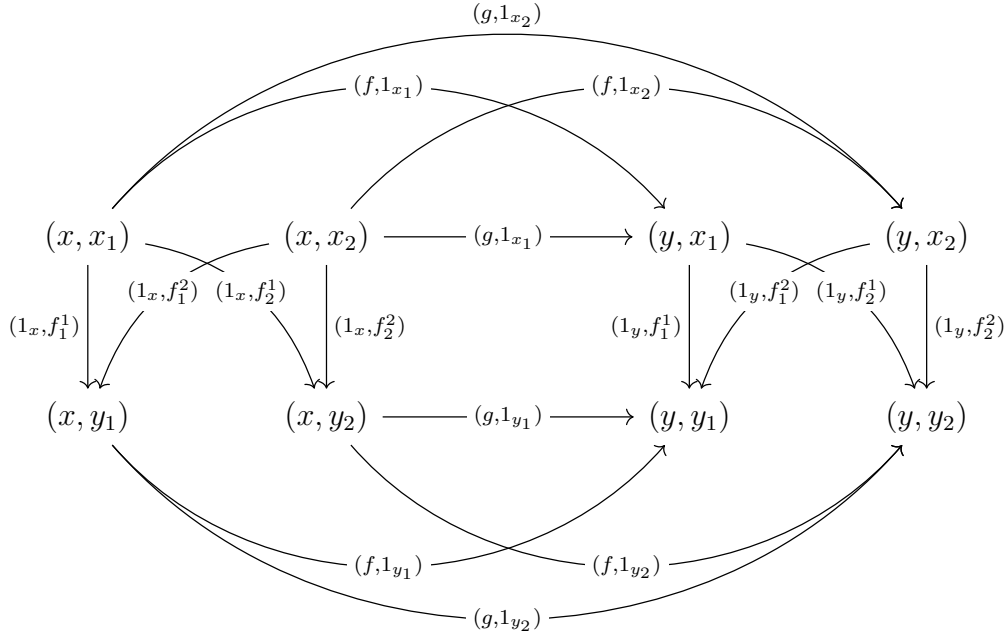
\[
\begin{tikzcd}[column sep=4em, row sep=4em]
{(x,x_1)} \arrow[d, "{(1_x,f_1^1)}"'] \arrow[rd, "{(1_x,f_2^1)}" description, bend left] \arrow[rrr, "{(f,1_{x_1})}" description, bend left=49] \arrow[rrrr, "{(g,1_{x_2})}", bend left=49] & {(x,x_2)} \arrow[ld, "{(1_x,f_1^2)}" description, bend right] \arrow[d, "{(1_x,f_2^2)}"] \arrow[rrr, "{(f,1_{x_2})}" description, bend left=49] \arrow[rr, "{(g,1_{x_1})}" description] &  & {(y,x_1)} \arrow[d, "{(1_y,f_1^1)}"'] \arrow[rd, "{(1_y,f_2^1)}" description, bend left] & {(y,x_2)} \arrow[ld, "{(1_y,f_1^2)}" description, bend right] \arrow[d, "{(1_y,f_2^2)}"] \\
{(x,y_1)} \arrow[rrr, "{(f,1_{y_1})}" description, bend right=49] \arrow[rrrr, "{(g,1_{y_2})}" description, bend right=49]                                                                  & {(x,y_2)} \arrow[rr, "{(g,1_{y_1})}" description] \arrow[rrr, "{(f,1_{y_2})}" description, bend right=49]                                                                               &  & {(y,y_1)}                                                                                & {(y,y_2)}                                                                               
\end{tikzcd}
\]
Here the vertical arrows are the unique morphisms of the poset \(S'\) (e.g., \((x,x_1)\to (x,y_1)\), etc.), and the horizontal arrows are the lifts of \(f\) and \(g\). 

The composition is determined by the Grothendieck construction, which gives the following relations:
\begin{itemize}
    \item \((f,1_{y_1})\circ (1_x,f_1^1)=(f,K'(f)(f_1^1))=(f,f_1^1)=(1_y,f_1^1)\circ (f,1_{x_1}).\)
    \item \((f,1_{y_1})\circ (1_x,f_1^2)=(f,K'(f)(f_1^2))=(f,f_1^2)=(1_y,f_1^2)\circ (f,1_{x_2}).\)
    \item \((g,1_{y_2})\circ (1_x,f_1^1)=(g,K'(g)(f_1^1))=(g,f_2^2)=(1_y,f_2^2)\circ (g,1_{x_2}).\)
    \item \((g,1_{y_2})\circ (1_x,f_1^2)=(g,K'(g)(f_1^2))=(g,f_2^1)=(1_y,f_2^1)\circ (g,1_{x_1}).\)
    \item \((f,1_{y_2})\circ (1_x,f_2^2)=(f,K'(f)(f_2^2))=(f,f_2^2)=(1_y,f_2^2)\circ (f,1_{x_2}).\)
    \item \((f,1_{y_2})\circ (1_x,f_2^1)=(f,K'(f)(f_2^1))=(f,f_2^1)=(1_y,f_2^1)\circ (f,1_{x_1}).\)
    \item \((g,1_{y_1})\circ (1_x,f_2^2)=(g,K'(g)(f_2^2))=(g,f_1^1)=(1_y,f_1^1)\circ (g,1_{x_1}).\)
    \item \((g,1_{y_1})\circ (1_x,f_2^1)=(g,K'(g)(f_2^1))=(g,f_1^2)=(1_y,f_1^2)\circ (g,1_{x_2}).\)
\end{itemize}

We now analyze sections. Fix the object \(x\) in \(\mathcal{S}\) and take \(f\) as the spanning tree. Hence \(K'\) is already in the correct form (constant fiber, trivial on the tree), and we can compute the monodromy as the homomorphism \(\pi_1(\mathcal{S},x) \cong \mathbb{Z} \to \operatorname{Aut}(S')\cong \mathbb{Z}_2\) given by \(1 \mapsto s'\). The strict gauge group is therefore \(\operatorname{Aut}(S')\cong \mathbb{Z}_2\) since the group is abelian.

This model has no strict fixed objects: \(s'\) swaps both \(x_1\leftrightarrow x_2\) and \(y_1\leftrightarrow y_2\), so no object is fixed. Hence, by the strict fixed point criterion, there are no horizontal sections. However, there are arbitrary (lax) sections. For example, choose \(s(x)=(x,x_1)\) and \(s(y)=(y,y_1)\). For \(f:x\to y\), we can choose $s(f)=(f,f_1^1)$ and for \(g:x\to y\), $s(g)=(g,f_1^2)$. Thus we obtain an arbitrary section. This shows that the absence of strict fixed points does not preclude the existence of arbitrary sections. It also illustrates the distinction between strict and lax fixed points.

\subsection{The projective plane covering}
Take the category $\mathcal{P}$ defined in Example \ref{ex:P}. Take the covering $E \colon \mathcal{P} \to \mathbf{cat}$ with fiber \(\{0,1\}\) and monodromy \(\sigma\) the permutation of two elements. The total category \(\mathcal{E}=\int E\) for the covering is:
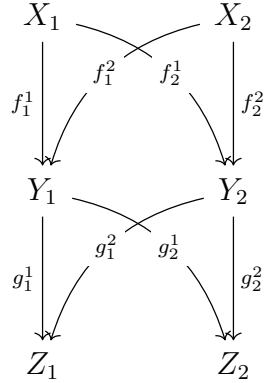
\[
\begin{tikzcd}[column sep=4em, row sep=4em]
X_1 \arrow[d, "f_1^1"'] \arrow[rd, "f_2^1" description, bend left] & X_2 \arrow[ld, "f_1^2" description, bend right] \arrow[d, "f_2^2"] \\
Y_1 \arrow[d, "g_1^1"'] \arrow[rd, "g_2^1" description, bend left] & Y_2 \arrow[ld, "g_1^2" description, bend right] \arrow[d, "g_2^2"] \\
Z_1                                                                & Z_2                                                               
\end{tikzcd}
\]
with the relations that transform it into a poset. One can check that the total category models a sphere \cite[Proposition 3.1.4]{MR3024764} and that $\mathcal{P}$ models the projective plane \cite[Example 3.8]{LSTan} using the classifying space.

The monodromy \(\sigma\) swaps 0 and 1, so there is no strict fixed point. Since the fiber is discrete, the notions of strict and lax fixed points coincide (there are no non-identity morphisms). Hence there is no section at all, neither horizontal nor arbitrary.

\subsection{Relative core reduction of a fiber bundle}

Take $\mathcal{S}$ the double arrow category and define the fiber bundle $P \colon \mathcal{S} \to \mathbf{cat}$ with fiber the category $\mathcal{D}$ defined in Example \ref{ex:beat-points} and where $P(f)=1_\mathcal{D}$ and $P(g)$ is the functor defined by the swap of $f_1$ with $f_3$. Then the total category has the following diagram:
$$
% https://tikzcd.yichuanshen.de/#N4Igdg9gJgpgziAXAbVABwnAlgFyxMJZABgBpiBdUkANwEMAbAVxiRAAoAPUugShAC+pdJlz5CKMgEYqtRizZdSAI35CR2PASJkATLPrNWiDtwDGa4SAybxRXeQPzjHAJ49LGsdpQOZ1QwUTdndVQSsbbwlkB30A50V3C0FZGCgAc3giUAAzACcIAFskMhAcCCQAZmplGDAoKtLAl3Yc0ikAfT4QagY6WoYABVEtCRAsMGxYcNyC4sRS8qqauobEAFpKpoTg9Pau-l7+mCGRuxMJqdZ1EHyipClqJcQAFhX6xvijRTbOsKOBsNbD5xpMsNMbnd5o8yhVXu81pttt9dvt-iA+oCziDLuDrlYoUgHLCkABWBGfOQoji-DrJAEnIFRNi4iEEuZEp5w8kgWofDZbL5BDh7Tr0jHHU7Asas-Gze4LLkPIUtTrcAAWhwlWOlLLBbPl82qJPhVOF7E67k1PW1jOxMv1ctuHMQMOexL5iIAbMjzWrSDkOlItZi7bqLo6Zs6FW64R7VkgfSrFP7A5UQ5KmaM9Vco4TXUrEMTmimOtxA7oMzrmRHc5CXW8TTyS8FLQGOpWbaGpTXQXX2QrG88eZ6kOsk2bVR13IHg13M-ac3i8w3CyOE4gJy2OG202oKAIgA
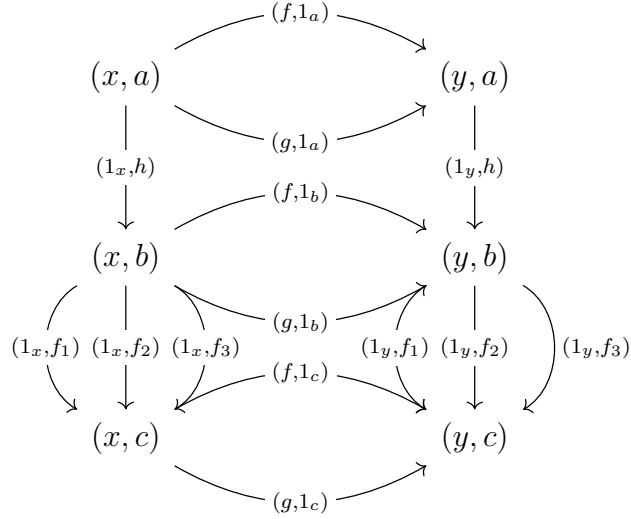
\begin{tikzcd}[column sep=4em, row sep=4em]
{(x,a)} \arrow[rr, "{(f,1_a)}" description, bend left] \arrow[rr, "{(g,1_a)}" description, bend right] \arrow[d, "{(1_x,h)}" description]                                                                                                          &  & {(y,a)} \arrow[d, "{(1_y,h)}" description]                                                                                              \\
{(x,b)} \arrow[rr, "{(f,1_b)}" description, bend left] \arrow[rr, "{(g,1_b)}" description, bend right] \arrow[d, "{(1_x,f_1)}" description, bend right=60] \arrow[d, "{(1_x,f_3)}" description, bend left=60] \arrow[d, "{(1_x,f_2)}" description] &  & {(y,b)} \arrow[d, "{(1_y,f_2)}" description] \arrow[d, "{(1_y,f_1)}" description, bend right=60] \arrow[d, "{(1_y,f_3)}", bend left=60] \\
{(x,c)} \arrow[rr, "{(f,1_c)}" description, bend left] \arrow[rr, "{(g,1_c)}" description, bend right]                                                                                                                                             &  & {(y,c)}                                                                                                                                
\end{tikzcd}
$$
where the composition is twisted by the action of $P(g)$ in the way indicated by the Grothendieck construction.

In this case we can check that $(x,b)$ is a beat object of the projection $\pi_P$ since it is a beat object and the connecting morphism is in the fiber. When we delete $(x,b)$ we can see that now $(y,b)$ is a beat object and finally we obtain a fiber bundle that is equivalent to the product $\mathcal{S} \times \mathcal{S}$. 

\subsection{A negligible object which is not a beat object}
\label{ex:negligible-not-beat}

Let \(\A\) be the category generated by
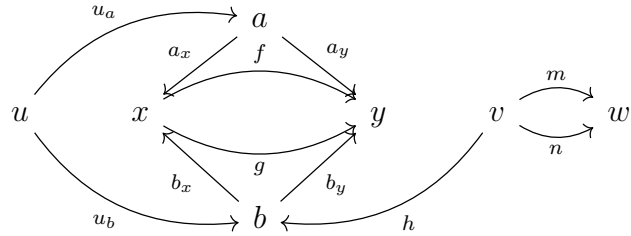
\[
% https://tikzcd.yichuanshen.de/#N4Igdg9gJgpgziAXAbVABwnAlgFyxMJZAJgBoAGAXVJADcBDAGwFcYkR6QBfU9TXfIRTlSARmp0mrds268QGbHgJFRYiQxZtEIAB5y+SwUQDM6mpuk6AngYX9lQ5ABZzkre1p3FAlSjLEGlLaIABG3g7GKACsbpYhAO7cEjBQAObwRKAAZgBOEAC2SCIgOBBIZO5WHAD6+jSM9KEwjAAKkX4guVhpABY4dnmFxTRlSGZVIfQ1tg1NLe1GnYww2QM8OflFiGql5Ygl8TI1nHPNbR1CICtrIDTNYFDj5BsgQ9u7Y4ixk8fhZwtLuxun0BvcYI8kABaEwveTvCqjfYTI46bJ3a7zC5LK43MFhCFPRCw14IxCVL4o4LsNIYxrnRa+K4g-oYh5EmFwzbDRCuPZIH6okC9OlYxmOdh4tmE56krYCpGI346UJ1UUMoE6KVynk-SkWakqmbqwE44E9Vk67Z8r4ANnBkOJh0NICKVqQNv29oJjs5Bo8OkIAOxTPNoOSXCAA
\begin{tikzcd}
                                                                &                                                           & a \arrow[ld, "a_x"'] \arrow[rd, "a_y"] &   &                                                                                     &   \\
u \arrow[rru, "u_a", bend left] \arrow[rrd, "u_b"', bend right] & x \arrow[rr, "f", bend left] \arrow[rr, "g"', bend right] &                                        & y & v \arrow[lld, "h", bend left] \arrow[r, "m", bend left] \arrow[r, "n"', bend right] & w \\
                                                                &                                                           & b \arrow[lu, "b_x"] \arrow[ru, "b_y"'] &   &                                                                                     &  
\end{tikzcd}
\]
subject to the relations
\[
a_xu_a=b_xu_b,
\qquad
a_yu_a=b_yu_b,
\]
together with
\[
fa_x=ga_x=a_y,
\qquad
fb_x=gb_x=b_y.
\]
No relation is imposed between \(m\) and \(n\) or \(f\) and \(g\). These relations also imply that all paths from \(u\) to \(y\) coincide. 

Set \(\A_u=\A\setminus\{u\}\), and let
\(\inc_u\co\A_u\to\A\) be the inclusion. We claim that \(u\) is upper
fundamental-groupoid negligible but is not a beat object.

Let \(u_x\co u\to x\) and \(u_y\co u\to y\) denote the common composites
\[
u_x=a_xu_a=b_xu_b,
\qquad
u_y=a_yu_a=b_yu_b=fu_x=gu_x.
\]
The comma category \(u\downarrow\inc_u\) has the four objects
\(
u_a,\text{} u_b,\text{} u_x,\text{}u_y.
\)
The nonidentity generating
morphisms of \(u\downarrow\inc_u\) are represented by
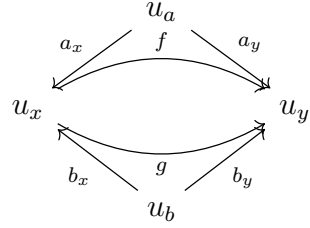
\[
\begin{tikzcd}
& u_a
  \arrow[ld, "a_x"']
  \arrow[rd, "a_y"]
& \\
u_x
  \arrow[rr, "f", bend left]
  \arrow[rr, "g"', bend right]
&&
u_y
\\
& u_b
  \arrow[lu, "b_x"]
  \arrow[ru, "b_y"']
&
\end{tikzcd}
\]
with relations
\[
fa_x=ga_x=a_y,
\qquad
fb_x=gb_x=b_y.
\]

We now compute its fundamental group. Choose \(a_x\), \(b_x\), and \(f\)
as a spanning tree. The remaining generators are \(g\), \(a_y\), and
\(b_y\). Since all tree morphisms represent the identity, the relations give
\[
[g]=[a_y]=[b_y]=1.
\]
Hence
\(\Pi_1(u\downarrow\inc_u)\) is equivalent to the terminal groupoid.
Therefore \(u\) is upper fundamental-groupoid negligible. By
Theorem~\ref{thm:pi1-negligible-removal}, the inclusion
\[
\inc_u\co\A_u\longrightarrow\A
\]
induces an equivalence of fundamental groupoids.

Nevertheless, \(u\) is not an up beat object. Indeed,
\(u\downarrow\inc_u\) has no initial object. The objects \(u_a\) and
\(u_b\) are distinct minimal objects, and there is no morphism between them.
Neither \(u_x\) nor \(u_y\) can be initial, since there are no morphisms from
either object to \(u_a\) or \(u_b\). Moreover,
\(\inc_u\downarrow u\) is empty, because no nonidentity morphism has
codomain \(u\). Thus \(u\) is not a down beat object either.

Consequently, \(u\) is fundamental-groupoid negligible but not a beat
object. We now describe the effect of removing it more explicitly.

In the category \(\A_u=\A\setminus\{u\}\), the object \(b\) becomes a
down beat object. Indeed, the only nonidentity morphism with codomain \(b\)
is \(h\co v\to b\). Hence \(h\) is the terminal object of
\(\widehat{\A_u\downarrow b}\). Let
\(\A_{u,b}=\A\setminus\{u,b\}\). Removing \(b\) therefore defines a
strong deformation retraction
\(\A_u\to\A_{u,b}\).

\(\A_{u,b}\) is generated by
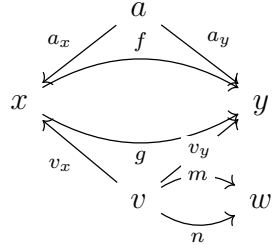
\[
% https://tikzcd.yichuanshen.de/#N4Igdg9gJgpgziAXAbVABwnAlgFyxMJZAJgBoAGAXVJADcBDAGwFcYkR6QBfU9TXfIRTlSARmp0mrds268QGbHgJFRYiQxZtEIAB5y+SwUQDM6mpuk6AngYX9lQkqWIap2uncUCVKM64t3dgB3bgkYKABzeCJQADMAJwgAWyQREBwIJDJJLXZ6AH19GkZ6ACMYRgAFB2MdBKxIgAscO0SUtJpMpDNcqw4C2xLyypqjXxBGGDjWnnik1MQc7sReyw84kGGK6tqJqZmtkAqwKB7yOZB2xeWs1cC8nUij0p2xnyEQBubWmhOzxAAWhMF3k1yQABYunccut2LQii8Rrtxp8DrMwQtIdCettRntPlgwNhYEc4ToEbZLuDEFCMncAKx-GCnc4PfqpPEoj7sIkktjUrG0nGIJnHFkA4HpcngJFvAnsb4tMJcIA
\begin{tikzcd}
  &                                                           & a \arrow[ld, "a_x"'] \arrow[rd, "a_y"]                                                                              &   \\ & x \arrow[rr, "f", bend left] \arrow[rr, "g"', bend right] &                                                                                                                     & y \\
  &                                                           & v \arrow[lu, "v_x"] \arrow[ru, "v_y" description] \arrow[r, "m" description, bend left] \arrow[r, "n"', bend right] & w
\end{tikzcd}
\]
subject to
\[
f\circ a_x=g\circ a_x=a_y,
\qquad
f\circ v_x=g\circ v_x=v_y,
\]
with no relation between \(m\) and \(n\). Since this is the wedge of the previous simply connected category with the category \(\mathcal{S}\) we have that the fundamental group is just \(\mathbb{Z}\).

\bibliographystyle{plain}
\bibliography{biblio}

\end{document}